\documentclass[a4paper,10pt,reqno]{amsart}

\usepackage[utf8]{inputenc}
\usepackage[T1]{fontenc}
\usepackage{amsthm}
\usepackage{amsmath}
\usepackage{amssymb}
\usepackage[inline]{enumitem}
\usepackage{comment}
\PassOptionsToPackage{hyphens}{url}\usepackage{hyperref}
\usepackage{fancyhdr}
\usepackage{mathrsfs}
\usepackage{stmaryrd}
\usepackage[normalem]{ulem}
\usepackage{xcolor}
\usepackage[skip=0pt]{caption}
\usepackage{tikz-cd}
\usetikzlibrary{arrows}

\theoremstyle{definition}

\newtheorem{theorem}{Theorem}[section]

\newtheorem{proposition}[theorem]{Proposition}
\newtheorem{corollary}[theorem]{Corollary}
\newtheorem{definition}[theorem]{Definition}
\newtheorem{example}[theorem]{Example}
\newtheorem{remark}[theorem]{Remark}

\definecolor{blue-url}{RGB}{0,0,100}
\definecolor{red-url}{RGB}{100,0,0}
\definecolor{green-url}{RGB}{0,100,0}
\definecolor{light-yellow}{RGB}{255,255,128}
\definecolor{light-blue}{RGB}{193,255,255}
\definecolor{light-red}{RGB}{239,83,80}

\hypersetup{
	pdftitle={On the finiteness of certain factorization invariants},
	pdfauthor={},
	pdfmenubar=false,
	pdffitwindow=true,
	pdfstartview=FitH,
	colorlinks=true,
	linkcolor=blue-url,
	citecolor=green-url,
	urlcolor=red-url
}

\renewcommand{\emptyset}{\varnothing}
\renewcommand{\setminus}{\smallsetminus}
\renewcommand{\,}{\kern 0.1em}

\providecommand\llb{\llbracket}
\providecommand\rrb{\rrbracket}
\providecommand\sqeq{\sqsubseteq}
\providecommand\sqneq{\sqsubset}
\providecommand{\RR}{\mathbin{R}}

\newcommand{\evid}[1]{\textsf{#1}}

{\newline\vspace{\abovedisplayskip}\hbox to \textwidth\bgroup\hss$\displaystyle}
{$\hss\egroup\vspace{\belowdisplayskip}}

\makeatletter
\DeclareFontFamily{OMX}{MnSymbolE}{}
\DeclareSymbolFont{MnLargeSymbols}{OMX}{MnSymbolE}{m}{n}
\SetSymbolFont{MnLargeSymbols}{bold}{OMX}{MnSymbolE}{b}{n}
\DeclareFontShape{OMX}{MnSymbolE}{m}{n}{
	<-6>  MnSymbolE5
	<6-7>  MnSymbolE6
	<7-8>  MnSymbolE7
	<8-9>  MnSymbolE8
	<9-10> MnSymbolE9
	<10-12> MnSymbolE10
	<12->   MnSymbolE12
}{}
\DeclareFontShape{OMX}{MnSymbolE}{b}{n}{
	<-6>  MnSymbolE-Bold5
	<6-7>  MnSymbolE-Bold6
	<7-8>  MnSymbolE-Bold7
	<8-9>  MnSymbolE-Bold8
	<9-10> MnSymbolE-Bold9
	<10-12> MnSymbolE-Bold10
	<12->   MnSymbolE-Bold12
}{}

\let\llangle\@undefined
\let\rrangle\@undefined
\DeclareMathDelimiter{\llangle}{\mathopen}%
{MnLargeSymbols}{'164}{MnLargeSymbols}{'164}
\DeclareMathDelimiter{\rrangle}{\mathclose}%
{MnLargeSymbols}{'171}{MnLargeSymbols}{'171}
\makeatother

\begin{document}

\title{On the finiteness of certain factorization invariants}
\author{Laura Cossu}
\address{(L.~Cossu) Department of Mathematics and Scientific Computing, University of Graz | Heinrichstrasse 36/III, 8010 Graz, Austria}
\email{laura.cossu@uni-graz.at}
\urladdr{https://sites.google.com/view/laura-cossu}
\author{Salvatore Tringali}
\address{(S.~Tringali) School of Mathematical Sciences,
Hebei Normal University | Shijiazhuang, Hebei province, 050024 China}
\email{salvo.tringali@gmail.com}
\urladdr{http://imsc.uni-graz.at/tringali}


\begin{abstract}
\noindent{}Let $H$ be a monoid, $\mathscr F(X)$ be the free monoid on a set $X$, and $\pi_H$ be the unique extension of the identity map on $H$ to a monoid homomorphism $\mathscr F(H) \to \allowbreak H$. Given $A \subseteq H$, an $A$-word $\mathfrak z$ (i.e., an element of $\mathscr F(A)$) is minimal if $\pi_H(\mathfrak z) \ne \allowbreak \pi_H(\mathfrak z')$ for every permutation $\mathfrak z'$ of a proper subword of $\mathfrak z$. The minimal $A$-elasticity of $H$ is then the supremum of all rational numbers $m/n$ with $m, n \in \allowbreak \mathbb N^+$ such that there exist minimal $A$-words $\mathfrak a$ and $\mathfrak b$ of length $m$ and $n$, resp., with $\pi_H(\mathfrak a) = \pi_H(\mathfrak b)$.

Among other things, we show that if $H$ is commutative  and $A$ is finite, then the minimal $A$-elasticity of $H$ is finite. This yields a non-trivial generalization of the finiteness part of a classical the\-o\-rem of Anderson et al.~from the case where $H$ is cancellative, commutative, and finitely generated (f.g.)~modulo units and $A$ is the set $\mathscr A(H)$ of its atoms. We also check that commutativity is somewhat essential here, by proving the existence of an atomic, cancellative, f.g.~monoid with trivial group of units whose minimal $\mathscr A(H)$-elasticity is infinite.
\end{abstract}

\maketitle \thispagestyle{empty}

\section{Introduction}
\label{sec:intro}

Let $H$ be a (multiplicatively written) monoid, e.g., the multiplicative monoid of a (unital, associative) ring. 
An \evid{irreducible} of $H$ is an element $a \in H$ that is neither a divisor of the identity $1_H$ nor a product of two other elements $x, y \in H$ each of which is neither a divisor of $1_H$ nor of $a$. An \evid{atom} of $H$ is, on the other hand, a non-unit that cannot be expressed as the product of two non-units. The existence itself of an atom implies that any divisor of $1_H$ is a unit \cite[Lemma 2.2]{Fa-Tr18}; hence, every atom is an irreducible, and the converse is true when the monoid is, e.g., commutative and cancellative (see Sect.~\ref{sec:closings}).

We define the (\evid{classical}) \evid{elasticity} of $H$ as the supremum (with respect to the standard ordering of the non-negative real numbers) of the set of rational numbers of the form $m/n$ with $m, n \in \mathbb N^+$ (the positive integers) such that $a_1 \cdots a_m = b_1 \cdots b_n$ for some irreducibles $a_1, \ldots, a_m, b_1, \ldots, b_n \in H$. Moreover, we say a monoid is \evid{atomic} if every non-unit has an \evid{atomic factorization}, that is, the element factors as a (finite) product of atoms.

Since the late 1980s, elasticity has received wide attention in the literature (see Anderson's survey \cite{DFAnd97} for an overview of results prior to 1997, \cite[Theorems 6.2 and 7.2]{Ka16} for some of the strongest finiteness criteria so far available in the cancellative commutative setting, and \cite{Ba-Co16, Ch-Sm06, Gry22, Tr19a, Zh19(a), Zh19(b)} for a non-exhaustive list of recent con\-tri\-bu\-tions). Introduced by Valenza \cite{Va90} in his study of factorization in number rings, the notion was made popular by Zaks \cite{Za76} who used it as a measure of the deviation of an atomic monoid from the condition of half-factoriality (a monoid is \evid{half-factorial} if it is atomic and any two atomic factorizations of the same element have the same \emph{length}, i.e., the same number of factors). Most notably, it is a classical result in the arithmetic theory of monoids (and rings) that the elasticity of a cancellative commutative monoid with finitely many non-associated atoms is a rational number and hence finite. (Two elements $u, v \in H$ are \evid{associated} if $u$ divides $v$ (namely, $v \in HuH$) and $v$ divides $u$.) This was first proved by Anderson et al.~in \cite[Theorem 7]{An-An-Ch-Sm93} and is herein referred to as \emph{Anderson et al.'s theorem}. The result has been later extended to unit-cancellative commutative monoids with finitely many non-associated atoms by Fan et al.~\cite[Proposition 3.4(1)]{FGKT}, where the monoid $H$ is called \evid{unit-cancellative} if $yx \ne x \ne xy$ for all $x, y \in H$ such that $y$ is a non-unit (obviously, a cancellative monoid is unit-cancellative).

In this article, we adopt a new point of view set forth in \cite{An-Tr18, Co-Tr-22(a)} and obtain a non-trivial generalization of the \emph{finiteness} part of Fan et al.'s result (and hence of Anderson et al.'s theorem) to any commutative monoid that is finitely generated up to units (Corollaries \ref{cor:minimal-elasticity-of-f.g.u.-comm-monoid} and \ref{cor:unit-cancellative}). The proof relies on \emph{Dickson's lemma} (see, e.g., Theorem 9.18 in \cite{Cl-Pr67}), which is the one and only aspect in common with previous results in the same vein. A critical feature of our approach is the use of \emph{minimal} fac\-tor\-i\-za\-tions to counter the blowup of factorization lengths and related invariants that is typical of a non-unit-cancellative or non-commutative setup (Examples~\ref{exa:blowup} and \ref{exa: non-comm counterexample}). The conclusion then becomes an immediate consequence of an es\-sen\-tial\-ly combinatorial theorem (Theorem \ref{thm:main}) that makes no reference to atoms, irreducibles, etc. 

The price we pay is that we cannot say anything about the rationality of the invariants we introduce along the way to generalize the classical elasticity (Sect.~\ref{sec:closings}). What we gain is that, in the very spirit of a series of recent papers by the same authors \cite{Tr20(c), Co-Tr-21(a), Co-Tr-22(a), Tr22(a)}, many of our results are no longer phrased in the language of \emph{monoids and irreducibles} (whose scope is, in a sense, too narrow) but rather in the more abstract language of \emph{monoids and preorders} (see Sect.~\ref{sec:premons} for details).
Loosely speaking, this allows us to use basically any sort of elements as ``building blocks'' in the factorization process (Proposition \ref{prop:reduction-to-irreds}).

Other than that, we prove a couple or so of finiteness results on length sets and their unions in a (commutative or non-commutative) monoid with finitely many irreducibles (Definition \ref{def:lengths&unions}, Propositions \ref{prop:elasticity} and \ref{prop:bounded word length}, and Corollary \ref{cor:unions-in-a-strongly-finite-permutable-F-space}), and we show by way of example that our main theorem (Theorem \ref{thm:main}) fails in the non-commutative setting (Sect.~\ref{sec:example}).

\vskip 0.05cm

\subsection*{Notation.} 
Through the paper, $\mathbb{N}$ is the set of non-negative integers, and for $a,b\in \mathbb N \cup \{\infty\}$ we let $\llb a, b\rrb := \allowbreak \{x\in \mathbb{N} \colon a\leq x\leq b\}$ be the \evid{discrete interval} from $a$ to $b$. 

A \evid{preorder} on a set $S$ is a reflexive and transitive binary relation on $S$. Given a preorder $\preceq$ on $S$, we write $x \prec y$ to mean that $x\preceq y$ and $y\not \preceq x$. (This convention also applies to the symbols $\sqeq$ and $\sqneq$, which we will likewise employ for preorders.) We say that $\preceq$ is \evid{artinian} if, for every $\preceq$-non-increasing sequence $(x_k)_{k\ge 0}$ in $S$, we have $x_k\preceq x_{k+1}$ for all but finitely many $k \in \mathbb N$.

We refer to \cite{Ho95} for generalities on monoids. In particular, we denote by $\mathscr F(X)$ the \evid{free monoid} on a set $X$ and refer to the elements of $\mathscr F(X)$ as \evid{$X$-words}. We use the symbols $\ast_X$ and $\varepsilon_X$, resp., for the operation and the identity of $\mathscr F(X)$. We take $\|\mathfrak u\|_X$ to be the (\evid{word}) \evid{length} of an $X$-word $\mathfrak u$, and for each $i \in \llb 1, \|\mathfrak u\|_X \rrb$ we let $\mathfrak u[i]$ be the $i^\text{th}$ letter of $\mathfrak u$. An $X$-word $\mathfrak v$ is then a (\evid{scattered}) \evid{subword} of $\mathfrak u$ if there is a strictly increasing function $\sigma \colon \llb 1, \|\mathfrak v\|_X \rrb \to \llb 1, \|\mathfrak u\|_X \rrb$ such that $\mathfrak u[\sigma(i)] = \mathfrak v[i]$ for each $i \in \llb 1, \allowbreak \|\mathfrak v\|_X \rrb$. When there is no serious risk of ambiguity, we drop the subscript ``$X$'' from the above notation. 

We use $H^\times$ for the group of units (or invertible elements) of a monoid $H$ and $\langle X \rangle_H$ for the submonoid of $H$ generated by a set $X \subseteq H$. We call $H$ \evid{reduced} if its only unit is the identity $1_H$ and we write $\pi_H$ for the unique extension of the identity map on $H$ to a monoid homomorphism $\mathscr{F}(H) \to H$.

\section{Elasticity}
\label{sec:elasticity}

By the definition given in the introduction, the (classical) elasticity of a monoid $H$ is the supremum of the set of all rational numbers of the form $\|\mathfrak b\|^{-1} \|\mathfrak a\|$ as $\mathfrak a$ and $\mathfrak b$ range over the non-empty $\mathscr I(H)$-words with $\pi_H(\mathfrak a) = \pi_H(\mathfrak b)$, where $\mathscr I(H)$ is the set of irreducibles of $H$. We aim to generalize this idea.

\begin{definition}
\label{def:permutable-preorder}
Given a set $X$, we denote by $\sqeq_X$  
the binary re\-la\-tion on the free monoid $\mathscr F(X)$ defined by $\mathfrak a \sqeq_X \mathfrak b$, for some $X$-words $\mathfrak a$ and $\mathfrak b$, if and only if $\mathfrak a$ is a \evid{permuted subword} of $\mathfrak b$, i.e., there is an injective function $\sigma \colon \llb 1, \|\mathfrak a\| \rrb \to \llb 1, \|\mathfrak b\| \rrb$ such that $\mathfrak a[i] = \mathfrak b[\sigma(i)]$.
\end{definition}

Since the composition of two injections is still an injection, it is immediate that the relation $\sqeq_X$ in Definition \ref{def:permutable-preorder} is a preorder on $\mathscr F(X)$. More precisely, $\sqeq_X$ is an artinian preorder, because $\mathfrak a \sqneq_X \mathfrak b$ implies $\|\mathfrak a\| < \|\mathfrak b\|$.
This makes it natural to talk about $\sqeq_X$-minimality in $\mathscr F(X)$. Moreover, the pair $(\mathscr F(X), \allowbreak \sqeq_X)$ is a \emph{strongly positive monoid} in the sense of \cite[Definition 2.3]{Co-Tr-22(a)}, meaning that (i) $\varepsilon \sqeq_X \mathfrak a$ for every $X$-word $\mathfrak a$ and (ii) $\mathfrak a \sqeq_X \mathfrak b$ yields $\mathfrak u \ast \mathfrak a \ast \mathfrak v \sqeq_X \mathfrak u \ast \mathfrak b \ast \mathfrak v$ for all $\mathfrak u, \mathfrak v \in \mathscr F(X)$, with the latter inequality being strict whenever so is the former.

\begin{definition}\label{def:minimal A factorizations}
Given a monoid $H$ and a set $A \subseteq H$, an $A$-word $\mathfrak z$ is \evid{minimal} if $\pi_H(\mathfrak z') \ne \allowbreak \pi_H(\mathfrak z)$ whenever $\mathfrak z' \sqneq_H \mathfrak z$. The \evid{minimal $A$-elasticity} $\varrho_A^\mathsf{m}(H)$ of $H$ is then the supremum of the set of all rational numbers of the form $m/n$ with $m, n \in \mathbb N^+$ such that there exist minimal $A$-words $\mathfrak a$ and $\mathfrak b$ with $\|\mathfrak a\| = \allowbreak m$, $\|\mathfrak b\| = n$, and $\pi_H(\mathfrak a) = \pi_H(\mathfrak b)$, where it is understood that $\sup \emptyset := 0$.
\end{definition}

With these preliminaries in place, we are just ready for the main theorem of the paper. In the proof, we will make use of Dickson's lemma \cite[Theorem 9.18]{Cl-Pr67}, stating that every non-empty subset of $\mathbb N^{\times n}$ ($n \in \allowbreak \mathbb N^+$) has at least one minimal element with respect to the \evid{product order} $\leq_n$ induced on $\mathbb{N}^{\times n}$ by the standard order on $\mathbb{N}$, meaning that $\mathfrak u \leq_n \mathfrak v$, for some $\mathfrak u, \allowbreak \mathfrak v \in \mathbb N^{\times n}$, if and only if $\mathfrak u[i] \le \mathfrak v[i]$ for each $i \in \allowbreak \llb 1, n \rrb$ (here, we regard $\mathfrak u$ and $\mathfrak v$ as $\mathbb N$-words of length $n$).

\begin{theorem}\label{thm:main}
The minimal $A$-elasticity of a commutative monoid $H$ is finite for every finite $A \subseteq H$.
\end{theorem}

\begin{proof}
Suppose $A$ is a finite subset of $H$ and put $s := |A| \in \mathbb N$. We may assume $s \ne 0$, or else $A$ is empty and the conclusion is trivial because the only minimal $A$-word is then the empty word $\varepsilon$. Accordingly, let $a_1, \ldots, a_s$ be an enumeration of $A$, and for each $t \in \llb 1, s \rrb$ denote by $\mathsf{v}_t$ the func\-tion $\mathscr F(A) \to \allowbreak \mathbb N$ that maps an $A$-word $\mathfrak a$ to its \evid{$a_t$-adic valuation}, i.e., to the number of indices $i \in \allowbreak \llb 1, \allowbreak \|\mathfrak a\| \rrb$ such that $a_t = \mathfrak a[i]$. It is trivial that
\begin{equation}\label{equ:norm-vs-valuations}
\|\mathfrak a\| = \mathsf{v}_1(\mathfrak a) + \cdots + \mathsf{v}_s(\mathfrak a),
\qquad \text{for all }\mathfrak a \in \mathscr F(A).
\end{equation}
Let $S$ be the set of all triples $(\mathfrak a, \mathfrak b, \mathfrak z)$ of $A$-words with $(\mathfrak a, \mathfrak b)\ne (\varepsilon, \varepsilon)$ and $\pi_H(\mathfrak a \ast \mathfrak z) = \pi_H(\mathfrak b \ast \mathfrak z)$ such that $\mathfrak a \ast \mathfrak z$ is a minimal $A$-word (the definition of $S$ is intentionally ``asymmetric'', insomuch as we do not require that also $\mathfrak b \ast \mathfrak z$ is a minimal $A$-word). 

If $(\mathfrak a, \mathfrak b, \mathfrak z) \in S$, then $\mathfrak b \ne \varepsilon$. Otherwise, $\mathfrak a$ would be a non-empty $A$-word; and since $(\mathscr F(A), \allowbreak \sqeq_H)$ is a strongly positive monoid (see the comments after Definition \ref{def:permutable-preorder}), we would find that $\mathfrak z$ is a proper subword of $\mathfrak a \ast \mathfrak z$ with $\pi_H(\mathfrak a \ast \mathfrak z) = \pi_H(\mathfrak z)$, contradicting that $\mathfrak a \ast \mathfrak z$ is a minimal $A$-word. So, it makes sense to define 
$$
\varrho^\ast := \sup \bigl\{ \|\mathfrak b\|^{-1} \|\mathfrak a\| \colon (\mathfrak a, \mathfrak b, \mathfrak z) \in S\}.
$$
It is clear that, if $(\mathfrak a, \mathfrak b)$ is a pair of minimal $A$-words with $\pi_H(\mathfrak a) = \allowbreak \pi_H(\mathfrak b)$ and $\mathfrak b \ne \allowbreak \varepsilon$, then $(\mathfrak a, \mathfrak b, \varepsilon) \in S$; and there is at least one such pair, since we can take $\mathfrak a = \mathfrak b = a_1$. It follows that $\varrho_A^\mathsf{m}(H) \le \varrho^\ast$. Consequently, in order to prove that $\varrho_A^\mathsf{m}(H)$ is finite, it suffices to check that $\varrho^\ast < \infty$.

For, let $\leq_{3s}$ be the product order induced on $\mathbb{N}^{\times 3s}$ by the standard order on $\mathbb{N}$
and $f$ be the function 
\[
S \to \mathbb{N}^{\times 3s} \colon (\mathfrak a, \mathfrak b, \mathfrak z) \mapsto (\mathsf v_1(\mathfrak a), \ldots, \mathsf v_s(\mathfrak a), \mathsf v_1(\mathfrak b), \ldots, \mathsf v_s(\mathfrak b), \mathsf v_1(\mathfrak z), \ldots, \mathsf v_s(\mathfrak z)).
\]
Since $A$ is non-empty, $f(S)$ is a non-empty subset of $\mathbb{N}^{\times 3s}$. Hence we gather from Dickson's lemma that the set $T$ of $\leq_{3s}$-min\-i\-mal elements of $f(S)$ is finite and non-empty. We thus see that
\[
R :=  \bigl\{ \|\mathfrak b\|^{-1} \|\mathfrak a\| \colon (\mathfrak a, \mathfrak b, \mathfrak z) \in f^{-1}(T) \bigr\}
\]
is a non-empty finite subset of $\mathbb{Q}_{\ge 0}$. In particular, it is straightforward from Eq.~\eqref{equ:norm-vs-valuations} that 
\[
\begin{split}
R & = \left\{\frac{\mathsf v_1(\mathfrak a) + \cdots + \mathsf v_s(\mathfrak a)}{\mathsf v_1(\mathfrak b) + \cdots + \mathsf v_s(\mathfrak b)} \colon (\mathfrak a, \mathfrak b, \mathfrak z) \in f^{-1}(T)\right\} \\
& = \left\{\frac{n_1 + \cdots + n_s}{d_1 + \cdots + d_s} \colon (n_1, \ldots, n_s, d_1, \ldots, d_s, v_1, \ldots, v_s) \in T\right\},
\end{split}
\]
so making it evident that $1 \le |R| \le |T| < \infty$. As a result, $R$ has a maximum element $r \in \mathbb{Q}_{\ge 0}$, and it is obvious that $r \le \rho^\ast$. We claim $r = \rho^\ast$ (of course, this will be enough to show that $\rho^\ast < \infty$).

Assume to the contrary that $r < \rho^\ast$. The set $\bar{S}$ of all triples $(\mathfrak a, \mathfrak b, \mathfrak z) \in S$ with $r < \|\mathfrak b\|^{-1} \|\mathfrak a\|$ is then a non-empty subset of $S \setminus f^{-1}(T)$; in particular, if $(\mathfrak a, \mathfrak b, \mathfrak z) \in f^{-1}(T)$, then $\|\mathfrak b\|^{-1} \|\mathfrak a\| \le r$. Let $\preceq_S$ be the binary relation on $S$ defined by $(\mathfrak a, \mathfrak b, \mathfrak z) \preceq_S (\mathfrak a', \mathfrak b', \mathfrak z')$, for some $(\mathfrak a, \mathfrak b, \mathfrak z), (\mathfrak a', \mathfrak b', \mathfrak z') \in S$, if and only if 
\[
\text{(i) } \|\mathfrak a\| + \|\mathfrak b\| < \|\mathfrak a'\| + \|\mathfrak b'\|, \quad \text{or} \quad
\text{(ii) }\|\mathfrak a\| + \|\mathfrak b\| = \|\mathfrak a'\| + \|\mathfrak b'\| \text{ and }\|\mathfrak z\| \le \|\mathfrak z'\|. 
\]
It is routine to verify that $\preceq_S$ is an artinian preorder on $S$. Since $\bar{S}$ is a non-empty subset of $S$, we thus obtain from the well-foundedness
of artinian preorders (see, e.g., Remark 3.9(3) in \cite{Tr20(c)}) that $\bar{S}$ has a $\preceq_S$-minimal element $\mathfrak p = (\mathfrak a, \mathfrak b, \mathfrak z)$. 
By construction, $\mathfrak p$ is in $S$ but not in $f^{-1}(T)$. In consequence, there exists $\mathfrak p_1 = \allowbreak (\mathfrak a_1, \allowbreak \mathfrak b_1, \mathfrak z_1) \in S$ such that $f(\mathfrak a_1, \mathfrak b_1, \mathfrak z_1) <_{3s} f(\mathfrak a, \mathfrak b, \mathfrak z)$, which means that 
\begin{equation}
\label{equ:triple-of-subwords}
\text{(j) }
\mathfrak a_1 \sqeq_H \mathfrak a, \ \mathfrak b_1 \sqeq_H \mathfrak b, \text{ and } \mathfrak z_1 \sqeq_H \mathfrak z,
\quad\text{and}\quad
\text{(jj) at least one of these inequalities is strict}. 
\end{equation}
Denote by $\approx_H$ the relation of $\sqeq_H$-equivalence on $\mathscr F(H)$, and for all $i \in \llb 1, s \rrb$ and $\mathfrak u, \mathfrak v \in \mathscr F(H)$ set $\Delta_i(\mathfrak u, \allowbreak \mathfrak v) := \allowbreak |\mathsf v_i(\mathfrak u) - \allowbreak \mathsf v_i(\mathfrak v)|$. 
We distinguish two cases depending on whether $\mathfrak b_1 \approx_H \mathfrak b$ or $\mathfrak b_1 \sqneq_H \mathfrak b$ (by the last dis\-played equation, there are no other possibilities). In each case, we will reach a contradiction, thus fin\-ish\-ing the proof of the theorem. 

\vskip 0.05cm

\textsc{Case 1:} $\mathfrak b_1 \approx_H \mathfrak b$. If $\mathfrak a_1 \approx_H \mathfrak a$, then $r < \|\mathfrak b\|^{-1} \|\mathfrak a\| = \|\mathfrak b_1\|^{-1} \|\mathfrak a_1\|$ (two $H$-words are $\sqeq_H$-equivalent only if they have the same length) and, by the second item of Eq.~\eqref{equ:triple-of-subwords}, $\mathfrak z_1 \sqneq_H \allowbreak \mathfrak z$; it follows that $\bar{S} \ni \allowbreak (\mathfrak a_1, \allowbreak \mathfrak b_1, \mathfrak z_1) \prec_S \mathfrak p$, which contradicts the $\preceq_S$-minimality of $\mathfrak p$ in $\bar{S}$. Therefore, we gather from the first item of Eq.~\eqref{equ:triple-of-subwords} that $\mathfrak a_1 \sqneq_H \mathfrak a$. Since $H$ is commutative and the equivalence $\mathfrak b_1 \approx_H \mathfrak b$ translates to $\mathfrak b_1$ and $\mathfrak b$ being the same $A$-word up to a permutation of their letters, it is then seen that 
\[
\pi_H(\mathfrak a_1 \ast \mathfrak z) =
\pi_H(\mathfrak a_1 \ast \mathfrak z_1 \ast \bar{\mathfrak z}_1) = 
\pi_H(\mathfrak a_1 \ast \mathfrak z_1) \, \pi_H(\bar{\mathfrak z}_1) = \pi_H(\mathfrak b \ast \mathfrak z_1) \, \pi_H(\bar{\mathfrak z}_1) = 
\pi_H(\mathfrak b \ast \mathfrak z) = \pi_H(\mathfrak a \ast \mathfrak z), 
\]
where $\bar{\mathfrak z}_1$ is the $A$-word $a_1^{\ast \Delta_1(\mathfrak z, \mathfrak z_1)} \ast \cdots \ast a_s^{\ast \Delta_s(\mathfrak z, \mathfrak z_1)}$.
Considering that $\mathfrak a_1 \ast \mathfrak z \sqneq_H \mathfrak a \ast \mathfrak z$, this is however in con\-tra\-dic\-tion with the fact that $\mathfrak a \ast \mathfrak  z$ is a minimal $A$-word.

\vskip 0.05cm

\textsc{Case 2:} $\mathfrak b_1 \sqneq_H \mathfrak b$. Let $\mathfrak a_2$ and $\mathfrak b_2$ be, resp., the $A$-words $a_1^{\ast \Delta_1(\mathfrak a, \mathfrak a_1)} \ast \cdots \ast a_s^{\ast \Delta_s(\mathfrak a, \mathfrak a_1)}$ and $a_1^{\ast \Delta_1(\mathfrak b, \mathfrak b_1)} \ast \cdots \ast a_s^{\ast \Delta_s(\mathfrak b, \mathfrak b_1)}$, and set $\mathfrak z_2 := \mathfrak a_1 \ast \mathfrak z$.
We have $\|\mathfrak b_1\| < \|\mathfrak b\|$ and hence $\mathfrak b_2 \ne \varepsilon$. 
Moreover, $\mathfrak a_2 \ast \mathfrak z_2 = \mathfrak a_2 \ast \mathfrak a_1 \ast \mathfrak z \approx_H \allowbreak \mathfrak a \ast \allowbreak \mathfrak z$ and $\pi_H(\mathfrak a_2\ast \mathfrak z_2) = \allowbreak \pi_H(\mathfrak a\ast \mathfrak z)$, since $H$ is commutative and $\mathfrak a_2\ast \mathfrak a_1=\mathfrak a$ up to a permutation of the letters. This shows that $\mathfrak a_2 \ast \mathfrak z_2$ is a minimal $A$-word. In addition,  
\[
\begin{split}
\pi_H(\mathfrak b_2 \ast \mathfrak z_2) = \pi_H(\mathfrak b_2 \ast \mathfrak a_1 \ast \mathfrak z) & = \pi_H(\mathfrak b_2 \ast \bar{\mathfrak z}_1) \, \pi_H(\mathfrak a_1 \ast \mathfrak z_1) = \pi_H(\mathfrak b_2 \ast \bar{\mathfrak z}_1) \, \pi_H(\mathfrak b_1 \ast \mathfrak z_1) \\ 
& = \pi_H(\mathfrak b_1 \ast \mathfrak b_2 \ast \mathfrak z_1 \ast \bar{\mathfrak z}_1) = \pi_H(\mathfrak b \ast \mathfrak z) = \pi_H(\mathfrak a \ast \mathfrak z)=\pi_H(\mathfrak a_2 \ast \mathfrak z_2),
\end{split}
\]
where $\bar{\mathfrak z}_1$ is the same as in \textsc{Case 1}.
It then follows that $(\mathfrak a_1, \mathfrak b_1, \mathfrak z_1)$ and $(\mathfrak a_2, \mathfrak b_2, \mathfrak z_2)$ are both in $S$, and  
\[
\|\mathfrak a_i\| + \|\mathfrak b_i\| \le \|\mathfrak a\| + \|\mathfrak b_i\| < \|\mathfrak a\| + \|\mathfrak b\| \qquad (i = 1, 2). 
\]
So, we get from the $\preceq_S$-minimality of $\mathfrak p$ in $\bar{S}$ that $\|\mathfrak b_i\|^{-1} \|\mathfrak a_i\| \le r$, and hence

\[r < \frac{\|\mathfrak a\|}{\|\mathfrak b\|} =\frac{\|\mathfrak a_1\|+ \|\mathfrak a_2\|}{\|\mathfrak b_1\| + \|\mathfrak b_2\|} \le \max\left\{\frac{\|\mathfrak a_1\|}{\|\mathfrak b_1\|}, \frac{\|\mathfrak a_2\|}{\|\mathfrak b_2\|}\right\} \le r, \]
which is however absurd. (For the second inequality in the last display, see, e.g.,~\cite[Lemma 1.41]{Bor-Poo-Sha-Zud14}.)
\end{proof}

It is not clear if, given a commutative monoid $H$ and a finite set $A \subseteq H$, the $A$-elasticity of $H$ is not only finite (as guaranteed by Theorem \ref{thm:main}), but also rational (cf.~Sect.~\ref{sec:closings}).

\section{Factorizations and [unions of] length sets}
\label{sec:premons}

Let $\mathcal{H}=(H,\preceq)$ be a \evid{premon}, i.e., a monoid $H$ paired with a preorder $\preceq$ on its underlying set (note that, in general, we require no compatibility between the operation in $H$ and the preorder). In particular, we write $\mid_H$ for the \evid{divisibility preorder} on $H$ (that is, $u \mid_H v$ if and only if $v \in H$ and $u \in HvH$) and $H^{\rm div}$ for the \evid{divisibility premon} $(H, \mid_H)$ of $H$.

Two elements $u, v \in H$ are \evid{$\preceq$-equivalent} if $u \preceq v \preceq u$, and $u$ is a \evid{$\preceq$-unit} (of $H$) if it is \evid{$\preceq$-equivalent} to the identity $1_H$ (otherwise, $u$ is a \evid{$\preceq$-non-unit}). A $\preceq$-non-unit $a \in \allowbreak H$ is then a \evid{$\preceq$-ir\-re\-duc\-i\-ble} if $a \ne xy$ for all $\preceq$-non-units $x, y \in H$ with $x \prec a$ and $y \prec a$. We use $\mathcal H^\times$ for the set of $\preceq$-units, and $\mathscr I(\mathcal H)$ for the set of $\preceq$-irreducibles of the monoid $H$, which we also refer to as the \evid{irreducibles} of the premon $\mathcal H$. These no\-tions were first considered in \cite[Definition 3.6]{Tr20(c)} and further studied in \cite{Co-Tr-21(a), Co-Tr-22(a), Tr22(a)}.

Following \cite[Sect.~3]{Co-Tr-22(a)}, we denote by $\sqeq_\mathcal{H}$  
the \evid{shuffling preorder} induced by $\preceq$, that is, the preorder on $\mathscr F(H)$ defined by $\mathfrak a \sqeq_\mathcal{H} \mathfrak b$, for some $H$-words $\mathfrak a$ and $\mathfrak b$, if and only if there is an injective function $\sigma \colon \llb 1, \allowbreak \|\mathfrak a\| \rrb \to \llb 1, \|\mathfrak b\| \rrb$ such that $\mathfrak a[i] \preceq \mathfrak b[\sigma(i)] \preceq \mathfrak a[i] $ for every $i \in \llb 1, \|\mathfrak a\| \rrb$. It turns out that, similar to the case of the pre\-order $\sqeq_H$ introduced in Definition \ref{def:permutable-preorder}, $\sqeq_\mathcal{H}$ is artinian. 

Accordingly, we let a \evid{$\preceq$-factorization} of an element $x \in H$ be an $\mathscr I(\mathcal H)$-word $\mathfrak a \in \allowbreak \pi_H^{-1}(x)$ and we set $
\mathcal{Z}_{\mathcal{H}}(x) := \allowbreak \pi_H^{-1}(x) \cap \allowbreak \mathscr F(\mathscr I(\mathcal H))$.
A \evid{minimal $\preceq$-factorization} of $x$ is then a $\sqeq_\mathcal{H}$-minimal word in $\mathcal{Z}_{\mathcal{H}}(x)$, i.e., an $\mathscr I(\mathcal H)$-word $\mathfrak a \in \allowbreak \pi_H^{-1}(x)$ such that there is no $\mathscr I(\mathcal H)$-word $\mathfrak b \in \allowbreak \pi_H^{-1}(x)$ with $\mathfrak b  \sqneq_\mathcal{H} \mathfrak a$. 
We denote the set of minimal $\preceq$-factorizations of $x$ by $\mathcal{Z}_{\mathcal{H}}^{\sf m}(x)$. It follows from the artinianity of $\sqeq_\mathcal{H}$ that $x$ has a $\preceq$-fac\-tor\-i\-za\-tion if and only if it has a minimal $\preceq$-fac\-tor\-i\-za\-tion (see Remark 3.3(1) in \cite{Co-Tr-22(a)}). 

\begin{definition}\label{def:elasticity}
\begin{enumerate*}[label=\textup{(\arabic{*})}, mode=unboxed]
\item Given a premon $\mathcal H=(H,\preceq)$, the \evid{$\preceq$-elasticity} $\varrho_\mathcal{H}(x)$ (resp., the \evid{minimal $\preceq$-elasticity} $\varrho_\mathcal{H}^\mathsf{m}(x)$) of an element $x \in H$ is the supremum of $\|\mathfrak b\|^{-1} \|\mathfrak a\|$ as $\mathfrak a$ and $\mathfrak b$ range over the non-empty $\preceq$-fac\-tor\-i\-za\-tions (resp., the non-empty minimal $\preceq$-factorizations) of $x$. (It is understood that $\sup\emptyset := 0$.)
\end{enumerate*}

\vskip 0.05cm

\begin{enumerate*}[label=\textup{(\arabic{*})}, mode=unboxed, resume]
\item The \evid{elasticity} $\varrho(\mathcal H)$ (resp., the \evid{minimal elasticity} $\varrho^\mathsf{m}(\mathcal H)$) of the premon $\mathcal H$ is then the supremum of $\varrho_\mathcal{H}(x)$ (resp., of $\varrho_\mathcal{H}^\mathsf{m}(x)$) as $x$ ranges over the $\preceq$-non-units of $H$. In particular, we let the elasticity (resp., the minimal elasticity) of the monoid $H$ be the elasticity (resp., the minimal elasticity) of $H^{\rm div}$.
\end{enumerate*}
\end{definition}
In the notation of Definition \ref{def:elasticity}, $\varrho_\mathcal{H}(x) = \varrho_\mathcal{H}^\mathsf{m}(x)= 0$ for every $x \in H$ whose set of $\preceq$-factorizations is empty. Also observe that $\varrho_\mathcal{H}^\mathsf{m}(1_H)=0$, since the only minimal $\preceq$-factorization of the identity $1_H$ is the empty word. Lastly, note that the elasticity of the monoid $H$ is nothing different from what we called the \emph{classical elasticity} of $H$ in Sect.~\ref{sec:intro}.

\begin{theorem}\label{thm:main-premon}
The minimal elasticity of a commutative premon with finitely many irreducibles is finite.
\end{theorem}

\begin{proof}
Let $\mathcal{H}=(H, \preceq)$ be a commutative premon with finitely many irreducibles and denote by $\varrho_{\rm irr}^\mathsf{m}(H)$ the minimal  $\mathscr{I}(\mathcal{H})$-elasticity $\varrho_{\rm irr}^\mathsf{m}(H)$ of $H$.
Since $\mathscr{I}(\mathcal{H})$ is a finite set (by hypothesis), we have from Theorem \ref{thm:main} that $\varrho_{\rm irr}^\mathsf{m}(H) < \infty$. This suffices to finish the proof, as it is immediate that a minimal $\preceq$-factorization of a $\preceq$-non-unit is a minimal $\mathscr{I}(\mathcal{H})$-word in the sense of Definition \ref{def:minimal A factorizations} and hence $\varrho^\mathsf{m}(\mathcal H) \le \varrho_{\rm irr}^\mathsf{m}(H)$.
\end{proof}

As mentioned in Sect.~\ref{sec:intro}, it was proved by Fan et al.~in \cite[Proposition 3.4(1)]{FGKT} that, if $H$ is a commutative unit-cancellative monoid $H$ such that the quotient $H/H^\times$ is finitely generated (i.e., $H$ is finitely generated modulo units), then the classical elasticity of $H$ is rational (and hence finite). In the next corollaries, we show how Theorem \ref{thm:main} can be used to recover the \emph{finiteness} part of Fan et al.'s result.

\begin{corollary}\label{cor:minimal-elasticity-of-f.g.u.-comm-monoid}
If a commutative monoid $H$ is finitely generated modulo units, then its minimal elasticity is finite.
\end{corollary}

\begin{proof}
By the commutativity of $H$, an element $u \in H$ divides the identity $1_H$ if and only if $u$ is a unit. Consequently, an irreducible of $H$ is a non-unit $a \in H$ that does not factor as a product of two non-units each of which is not divisible by $a$; and the irreducibles of the quotient monoid $H/H^\times$ are nothing but the cosets of the irreducibles of $H$. It follows that $H$ and $H/H^\times$ have the same minimal elasticity. On the other hand, since $H$ is finitely generated modulo units if and only if $H/H^\times$ is finitely generated, we may assume without loss of generality that $H$ is reduced. By \cite[Example 2.5(1) and Theorem 4.7]{Co-Tr-22(a)}, the set of irreducibles of $H$ is then finite; and by Theorem \ref{thm:main}, this suffices to complete the proof.
\end{proof}

\begin{corollary}\label{cor:unit-cancellative}
If a commutative, unit-cancellative monoid is finitely generated modulo units, then its \textup{(}classical\textup{)} elasticity is finite.
\end{corollary}

\begin{proof}
This follows from Corollary \ref{cor:minimal-elasticity-of-f.g.u.-comm-monoid} when considering that, in a commutative and unit-cancellative monoid $H$, every irreducible (i.e., every $\mid_H$-irreducible) is an atom \cite[Corollary 4.4]{Tr20(c)} and every minimal $\mid_H$-fac\-tor\-i\-za\-tion is a $\mid_H$-factorization \cite[Proposition 4.7(v)]{An-Tr18}.
\end{proof}

The following definitions generalize sets of lengths and related invariants from the classical theory of factorization, see \cite[Sects.~2.3 and 4.1]{An-Tr18} and \cite[Definition 3.2]{Co-Tr-22(a)}.

\begin{definition}\label{def:lengths&unions} 
\begin{enumerate*}[label=\textup{(\arabic{*})}, mode=unboxed]
\item\label{def:lengths&unions(1)} Given a premon $\mathcal H = (H, \preceq)$ and an element $x \in H$, we let 
\[
\mathsf{L}_{\mathcal{H}}(x) := {\bigl\{ \|\mathfrak{a}\|_H: \mathfrak{a} \in \mathcal{Z}_{\mathcal{H}}(x) \bigr\}} \subseteq \mathbb N
\quad\text{and}\quad
\mathsf{L}_{\mathcal{H}}^{\sf m}(x) := {\bigl\{ \|\mathfrak{a}\|_H: \mathfrak{a} \in \mathcal{Z}_{\mathcal{H}}^{\sf m}(x) \bigr\}} \subseteq \mathbb N
\]
be, resp., the \evid{length set} and the \evid{minimal length set} of $x$ (relative to $\mathcal H$). Accordingly, we refer to
\[
\mathscr{L}(\mathcal{H}) := \bigl\{ \mathsf{L}_{\mathcal{H}}(x) \colon x\in H\setminus \mathcal{H}^\times \bigr\}
\quad\text{and}\quad
\mathscr{L}^{\sf m}(\mathcal{H}) := \bigl\{ \mathsf{L}_{\mathcal{H}}^{\sf m}(x) \colon x\in H\setminus \mathcal{H}^\times \bigr\},
\]
resp., as the \evid{system of length sets} and the \evid{system of minimal length sets} of $\mathcal H$; and given $k \in \mathbb N$, we call
\[
\mathscr U_k(\mathcal H) := \bigcup \,\{ L \in \mathscr L(\mathcal H) \colon k \in L\}
\qquad\text{and}\qquad
\mathscr U_k^{\mathsf{m}}(\mathcal H) := \bigcup \,\{ L \in \mathscr L^\mathsf{m}(\mathcal H) \colon k \in L\}
\]
resp., the \evid{union of length sets} and the \evid{union of minimal length sets} containing $k$.
\end{enumerate*}

\vskip 0.05cm

\begin{enumerate*}[label=\textup{(\arabic{*})}, mode=unboxed, resume]
\item In particular, we take the system of length sets (resp., the system of minimal length sets) of the monoid $H$ to be the system of length sets (resp., of minimal length sets) of the divisibility premon $H^{\rm div}$ of $H$, and we write $\mathscr L(H)$ for $\mathscr L(H^{\rm div})$ and $\mathscr L^{\sf m}(H)$ for $\mathscr L^{\sf m}(H^{\rm div})$. The same goes with [minimal] length sets and unions of [minimal] length sets.
\end{enumerate*}
\end{definition}

Note that, in the notation of Definition \ref{def:lengths&unions}, the sets $\mathscr U_0(\mathcal H)$ and $ \mathscr U_0^{\sf m}(\mathcal H)$ are both empty, because there is no $\preceq$-non-unit whose length set contains $0$ (in fact, $0$ is only a $\preceq$-length of the identity $1_H$).

\begin{proposition}\label{prop:elasticity}
The following hold for a premon $\mathcal H = (H, \preceq)$:

\begin{enumerate}[label=\textup{(\roman{*})}]
\item\label{prop:elasticity 1} 
$\varrho(\mathcal H) \le 1$ (resp., $\varrho^\mathsf{m}(\mathcal H) \le 1$) if and only if $|\mathsf L_\mathcal{H}(x)| \le 1$ (resp., $|\mathsf L_\mathcal{H}^{\sf m}(x)| \le 1$) for each $x\in H\setminus \mathcal{H}^\times$.
\item\label{prop:elasticity 2} If $\varrho(\mathcal H)$ (resp., $\varrho^\mathsf{m}(\mathcal H)$) is finite, then $\mathscr U_k(\mathcal H)$ (resp., $\mathscr U_k^{\sf m}(\mathcal H)$) is finite for every $k \in \mathbb N$.
\end{enumerate}
\end{proposition}

\begin{proof}
We focus on length sets and leave the corresponding statements on \emph{minimal} length sets to the reader (the proofs are pretty much the same in either case).

\vskip 0.05cm

\ref{prop:elasticity 1} If $|\mathsf{L}_\mathcal{H}(x_0)| \le 1$ for every $\preceq$-non-unit $x \in H$, then $\varrho_\mathcal{H}(x)$ is either $0$ or $1$, and hence $\varrho(\mathcal H) \le 1$. As for the converse, assume $\varrho(\mathcal H) \le 1$ and suppose by way of contradiction that there is a $\preceq$-non-unit $x_0 \in \allowbreak H$ such that $|\mathsf{L}_\mathcal{H}(x_0)| \ge 2$. There exist $k_1, k_2 \in \mathsf{L}_\mathcal{H}(x_0)$ with $1 \le k_1 < k_2$, implying that $\varrho(\mathcal H) \ge \allowbreak \varrho_\mathcal{H}(x_0) \ge \allowbreak k_2/k_1 > 1$, which is absurd.

\vskip 0.05cm

\ref{prop:elasticity 2} Let $\varrho(\mathcal H)$ be finite and suppose for a contradiction that $|\mathscr{U}_k(\mathcal H)| = \infty$ for some $k \in \mathbb N$. Since $\mathscr{U}_0(\mathcal H)$ is empty, $k$ is then a \emph{positive} integer and there is a sequence $(\mathfrak a_1, \mathfrak b_1), \allowbreak (\mathfrak a_2, \mathfrak b_2), \ldots$ of pairs of $\mathscr I(\mathcal H)$-words such that, for every $i \in \mathbb N^+$, $\mathfrak a_i$ and $\mathfrak b_i$ are $\preceq$-factorizations of the same $\preceq$-non-unit $x_i \in H$ and $k = \allowbreak \|\mathfrak a_i\| \le \allowbreak \|\mathfrak b_i\|<\|\mathfrak b_{i+1}\|$. This, however, contradicts the finiteness of $\varrho(\mathcal H)$.
\end{proof}

\begin{proposition}\label{prop:bounded word length}
Let $\mathcal H=(H,\preceq)$ be a premon. Then either there is a bound $M \in \mathbb{N}$ such that $\|\mathfrak a\| \le M$ for every minimal $\preceq$-factorization $\mathfrak a$, or for each $k \in \mathbb{N}$ there is a $\preceq$-minimal factorization of length $k$.
\end{proposition}
\begin{proof}
Suppose that there is an integer $k \ge 1$ such that the set of minimal $\preceq$-factorizations of length $k$ is empty. We claim that there is no minimal $\preceq$-factorization of length $n \ge k$, and we proceed to prove it by induction on $n$. If $n = k$, the assertion is obvious. Otherwise, let $\mathfrak a = a_1 \ast \cdots \ast a_n$ be an $\mathscr{I}(\mathcal{H})$-word of length $n \ge k+1$ and assume inductively that there is no minimal $\preceq$-factorization of length $n-1$. In particular, this means that $\mathfrak a' := a_1 \ast \cdots \ast a_{n-1}$ is not a minimal $\preceq$-factorization, viz., there exists an $\mathscr{I}(\mathcal{H})$-word $\mathfrak b$ with $\mathfrak b \sqneq_\mathcal{H} \mathfrak a'$ and $\pi_H(\mathfrak b)=\pi_H(\mathfrak a')$. Since $(\mathscr{F}(\mathscr{I}(\mathcal{H})),\sqeq_\mathcal{H})$ is a strongly positive monoid (as noted in the second paragraph of this section), it follows that $\mathfrak b \ast a_n \sqneq_\mathcal{H} \mathfrak a$; and since  $\pi_H(\mathfrak b \ast a_n) = \allowbreak \pi_H(\mathfrak a)$, we conclude that $\mathfrak a$ is not a minimal $\preceq$-factorization either.
\end{proof}

In \cite[Proposition 2 and Corollary 1]{Ge-Le90}, Geroldinger and Lettl proved that the unions of length sets (of the divisibility premon) of a cancellative, commutative, finitely generated monoid $H$ are all finite. This result is generalized to a non-commutative, non-cancellative context by Corollary \ref{cor:unions-in-a-strongly-finite-permutable-F-space} below.

\begin{theorem}\label{thm:local-BF-ness}
Let $\mathcal{H}=(H,\preceq)$ be a premon and suppose there is a finite set $A \subseteq \mathscr I(\mathcal H)$ such that every $\preceq$-irreducible is $\preceq$-equivalent to an element of $A$. Then the minimal length sets of $\mathcal H$ are all finite.
\end{theorem}

\begin{proof}
Denote by $\sim_\mathcal{H}$ the relation of $\preceq$-equivalence on $H$ and by $\approx_\mathcal{H}$ the relation of $\sqeq_\mathcal{H}$-equivalence on $\mathscr{F}(H)$, and  assume without loss of generality that the elements of $A$ are pairwise $\preceq$-inequivalent. 
Next, let $\{q_1, \ldots, q_s\}$ be an enumeration of the elements of $A$, where $s := |A| \in \mathbb{N}^+$; and for each $i \in \llb 1, s \rrb$, let $\mathsf{v}_i \colon \mathscr{F}(\mathscr{I}(\mathcal{H})) \to \mathbb{N}$ be the function that maps the empty word to $0$ and a non-empty $\mathscr{I}(\mathcal{H})$-word $a_1 \ast \allowbreak \cdots \ast a_n$ of length $n$ to the number of indices $j \in \llb 1, n \rrb$ such that $a_j \sim_\mathcal{H} q_i$. 

By assumption, we are given that, for every $a\in \mathscr{I}(\mathcal{H})$, there is a unique $i\in \llb 1, s \rrb$ such that $a\sim_\mathcal{H} q_i$. Moreover, $q_i \sim_\mathcal{H} q_j$, for some $i, j \in \llb 1, s \rrb$, if and only if $i = j$. It is then readily seen that

\begin{equation}\label{equ:norm-vs-valuations(1)}
\|\mathfrak a\| = \mathsf{v}_1(\mathfrak a) + \cdots + \mathsf{v}_s(\mathfrak a),
\qquad \text{for all } \mathfrak a \in \mathscr F(\mathscr I(\mathcal H)).
\end{equation}

Now, fix $x\in H$. We need to show that the minimal length set $\mathsf L_\mathcal{H}^{\sf m}(x)$ of $x$ is finite. If $\mathsf L_\mathcal{H}^{\sf m}(x)$ is empty, the conclusion is obvious. So, suppose that $x$ has at least one minimal $\preceq$-factorization, and let $\leq_s$ be the product order induced on $\mathbb{N}^{\times s}$ by the standard order on $\mathbb N$ and $f$ be the function 
\[
\mathsf{Z}_\mathcal{H}^{\sf m}(x) \to \mathbb{N}^{\times s} \colon \mathfrak a \mapsto (\mathsf v_1(\mathfrak a), \ldots, \mathsf v_s(\mathfrak a)). 
\]
Since $f(\mathsf{Z}_\mathcal{H}^{\sf m}(x))$ is a non-empty subset of $\mathbb{N}^{\times s}$, we get from  Dickson's lemma that the set $\mathcal M$ of $\leq_{s}$-min\-i\-mal elements of $f(\mathsf{Z}_\mathcal{H}^{\sf m}(x))$ is finite and non-empty. We claim that $f(\mathfrak a) \in \mathcal M$ for every $\mathfrak a \in \mathsf{Z}_\mathcal{H}^{\sf m}(x)$; note that this will finish the proof, as it implies by Eq.~\eqref{equ:norm-vs-valuations(1)} that
\[
|\mathsf L_\mathcal{H}^{\sf m}(x)| = |\{\mathsf v_1(\mathfrak a) + \cdots + \mathsf v_s(\mathfrak a) \colon \mathfrak a \in \mathsf{Z}_\mathcal{H}^{\sf m}(x)\}| = |\{k_1 + \cdots + k_s \colon (k_1, \ldots, k_s) \in \mathcal M\}| \le |\mathcal M| < \infty. 
\]
For the claim, suppose to the contrary that there exists $\mathfrak a \in \mathsf{Z}_\mathcal{H}^{\sf m}(x)$ with $f(\mathfrak a) \notin \mathcal M$. Then $f(\mathfrak b) <_s f(\mathfrak a)$ for some $\mathfrak b \in f^{-1}(\mathcal M)$, meaning that $\mathsf v_i(\mathfrak b) \le \mathsf v_i(\mathfrak a)$ for each $i \in \llb 1, s \rrb$ and at least one of these in\-e\-qual\-i\-ties is strict. By the definition of $\sqeq_\mathcal{H}$, it then follows that
\[
\mathfrak b \approx_\mathcal{H} q_1^{\mathsf v_1(\mathfrak b)} \ast \cdots \ast q_s^{\mathsf v_s(\mathfrak b)} \sqeq_\mathcal{H}  q_1^{\mathsf v_1(\mathfrak a)} \ast \cdots \ast q_s^{\mathsf v_s(\mathfrak a)} \approx_\mathcal{H} \mathfrak a,
\]
which, in turn, yields $\mathfrak b \sqneq_\mathcal{H} \mathfrak a$, because two $\mathscr{I}(\mathcal{H})$-words are $\sqeq_\mathcal{H}$-equivalent only if they have the same length and, by Eq.~\eqref{equ:norm-vs-valuations(1)} and the above, $\|\mathfrak b\| < \|\mathfrak a\|$. We have thus reached a contradiction, because $\mathfrak a$ and $\mathfrak b$ are both minimal $\preceq$-factorizations of $x$.
\end{proof}

\begin{corollary}\label{cor:unions-in-a-strongly-finite-permutable-F-space}
In a premon $\mathcal H = (H, \preceq)$ with finitely many $\preceq$-irreducibles, unions of minimal length sets are all finite.
\end{corollary}

\begin{proof}
Assume for a contradiction that $|\mathscr U^{\mathsf{m}}_k(\mathcal{H})| = \infty$ for some $k \in \mathbb{N}$. There is then a sequence $(\mathfrak a_1, \mathfrak b_1), \allowbreak (\mathfrak a_2, \mathfrak b_2), \ldots$ of pairs of $\mathscr I(\mathcal{H})$-words such that, for each $i \in \mathbb{N}^+$, $\mathfrak a_i$ and $\mathfrak b_i$ are minimal $\preceq$-fac\-tor\-i\-za\-tions of one same element $x_i \in H\setminus \mathcal{H}^\times$ with $k = \|\mathfrak a_i\| \le \|\mathfrak b_i\| < \|\mathfrak b_{i+1}\|$. However, the set of $\mathscr I(\mathcal{H})$-words of length $k$ is finite, since the basis $\mathscr I(\mathcal{H})$ is finite. So, we can find a (strictly) increasing sequence $i_1, i_2, \ldots$ of positive in\-te\-gers with $\mathfrak a_{i_1} = \mathfrak a_{i_j}$ for all $j \in \mathbb{N}^+$, implying that the minimal length set of $x_{i_1}$ is unbounded (as it contains the increasing sequence $\|\mathfrak b_{i_1}\|, \|\mathfrak b_{i_2}\|, \ldots$) and hence contradicting Theorem \ref{thm:local-BF-ness}.
\end{proof}

The next proposition shows that, under mild conditions on the premon $\mathcal H$ (see Remark \ref{rem:examples}) and up to a suitable modification of Definitions \ref{def:elasticity} and \ref{def:lengths&unions}, the results of the present section can be extended to the case when the factors used in the factorization process are taken from an arbitrary set $A \subseteq \mathscr{I}(\mathcal{H})$.

\begin{proposition}\label{prop:reduction-to-irreds}
Let $\mathcal H = (H, \preceq)$ be a premon in which the product of any two $\preceq$-non-units is a $\preceq$-non-unit, and let $A$ be a set of $\preceq$-irreducibles. There then exists a preorder $\preceq_A$ on $H$ such that each of the following conditions is satisfied:
\begin{enumerate}[label=\textup{(\roman{*})}]
\item\label{prop:reduction-to-irreds(i)} $u \in H$ is a $\preceq_A$-non-unit if and only if $u \in \langle A \rangle_H \setminus \mathcal H^\times$.
\item\label{prop:reduction-to-irreds(ii)} $a \in H$ is a $\preceq_A$-irreducible if and only if $a \in A$.
\item\label{prop:reduction-to-irreds(iii)} $a \preceq b \preceq a$, for some $a, b \in A$, if and only if $a \preceq_A b \preceq_A a$.
\end{enumerate}
\end{proposition}

\begin{proof}
Set $S := \langle A \rangle_H \setminus \mathcal H^\times$ and let $\phi$ be the function $H \to \mathbb N$ that maps an element $x \in S$ to the smallest integer $n \ge 1$ such that $x \in A^n$ and an element in $H \setminus S$ to $0$. We define a binary relation $\preceq_A$ on $H$ by taking $x \preceq_A y$ if and only if one of the following conditions is satisfied:
$$
(1)\ x, y \in H \setminus S,
\quad (2)\  x, y \in S \text{ and } \phi(x) < \phi(y),
\quad\text{or}\quad (3)\ x, y \in S, \ 
\phi(x) = \phi(y), \text{ and }x \preceq y. 
$$
It is routine to check that $\preceq_A$ is a preorder, so we focus below on proving that $\preceq_A$ satisfies \ref{prop:reduction-to-irreds(i)}--\ref{prop:reduction-to-irreds(iii)}.

\vskip 0.15cm

\ref{prop:reduction-to-irreds(i)} It is clear that $u \in H$ is a $\preceq_A$-unit if and only if $u \notin S$ (by the fact that $1_H \notin S$). Consequently, an element of $H$ is a $\preceq_A$-non-unit if and only if it belongs to $S$.

\vskip 0.15cm

\ref{prop:reduction-to-irreds(ii)} Fix $a \in A$ and assume for a contradiction that $a$ is not a $\preceq_A$-irreducible. Then, $a=bc$ for some $b,c\in S$ (i.e., $\preceq_A$-non-units) such that $b\prec_A a$ and $c\prec_A a$. Since $\phi(a)=1$, it must be $\phi(b)=\phi(c)=1$ and $b\prec a$ and $c\prec a$. Since $b,c\notin \mathcal{H}^\times$, we contradict that $a$ is a $\preceq$-irreducible and we prove that every $a\in A$ is a $\preceq_A$-irreducible. On the other hand, let $a\in H$ be a $\preceq_A$-irreducible and assume for a contradiction that $a\notin A$. Since $a$ is a $\preceq_A$-non-unit, then $a\in S$ and there exist $n\ge 2$ elements $a_1,\dots, a_n\in A$ such that $a=a_1 \cdots a_n$. Thus $a=bc$, with $b:=a_1$ and $c:=a_2\cdots a_n$. But this contradicts that $a$ is a $\preceq_A$-irreducible, since $b$ and $c$ are both in $S$ by the fact that a non-empty product of $\preceq$-non-units is still a $\preceq$-non-unit.

\vskip 0.15cm

\ref{prop:reduction-to-irreds(iii)} Pick $a, b \in A$. It is clear that $\phi(a) = \phi(b) = 1$. Since $A$ is a set of $\preceq$-irreducibles and a $\preceq$-ir\-re\-duc\-i\-ble is a $\preceq$-non-unit, $a$ and $b$ are both elements of $S$. Therefore, the definition itself of the preorder $\preceq_A$ implies that $a \preceq_A b \preceq_A a$ if and only if $a \preceq b \preceq a$.
\end{proof}

\begin{remark}\label{rem:examples}
\begin{enumerate*}[label=\textup{(\arabic{*})}]
\item Let $H$ be a monoid. If $x, y\in H$ and $xy \mid_H 1_H$ (i.e., $xy$ is a $\mid_H$-unit), then also $x$ and $y$ divide $1_H$; that is, the product of any two $\mid_H$-non-units is still a $\mid_H$-non-unit and Proposition \ref{prop:reduction-to-irreds} applies to the divisibility premon $(H, \mid_H)$ of $H$. The same is true of the premons $(H, \vdash_H)$ and $(H, \dashv_H)$, where $\vdash_H$ and $\dashv_H$ are, resp., the ``divides from the left'' and the ``divides from the right'' preorders on $H$ (i.e., $x \vdash_H y$ if and only if $x \in H$ and $y \in xH$, and similarly for $\dashv_H$).
\end{enumerate*}

\vskip 0.05cm

\begin{enumerate*}[label=\textup{(\arabic{*})}, resume]
\item Following \cite[Definition 2.3]{Co-Tr-22(a)}, let a \evid{weakly positive monoid} be a premon $\mathcal H = (H, \preceq)$ such that $1_H \preceq \allowbreak x$ and $uxv \preceq x \preceq yxz$ for all $x, y, z \in H$ and $u, v \in \mathcal H^\times$. By \cite[Remark 2.4(3)]{Co-Tr-22(a)} , the product of two $\preceq$-non-units is then a $\preceq$-non-unit. That is, Proposition \ref{prop:reduction-to-irreds} applies also to weakly positive monoids.
\end{enumerate*}
\end{remark}

We conclude with a couple of examples that show how the paradigm of \emph{minimal} factorizations (as opposite to ``ordinary factorizations'') can indeed mitigate the effects of blowup phenomena that would otherwise affect the invariants considered through this section, making them lose most of their significance.

\begin{example}\label{exa:blowup}
\begin{enumerate*}[label=\textup{(\arabic{*})}, mode=unboxed]
\item Let $H$ be the multiplicative monoid of the integers modulo $p^n$,
where $p \in \mathbb{N}$ is a prime and $n$ is an integer $\ge 2$. By \cite[Example 3.4]{Co-Tr-22(a)}, $H$ is an atomic monoid and the atoms (resp., the units) of $H$ are precisely the $\mid_H$-irreducibles (resp., the $\mid_H$-units). In addition, every non-zero non-unit of $H$ has an essentially unique atomic factorization, with ``essentially unique'' meaning that any two atomic fac\-tor\-i\-za\-tions of the same element are equivalent with respect to the shuffling preorder induced by the divisibility preorder $\mid_H$. On the other hand, the residue class of $0$ modulo $p^n$ has an essentially unique minimal $\mid_H$-factorization (of length $n$), but atomic factorizations of any length $\ge n$. It follows that the (classical) elasticity of $H$ is $\infty$; the minimal elasticity is $1$; and for every $k \in \mathbb N^+$, we have
\begin{equation*}
\mathscr U_k(H) = 
\left\{
\begin{array}{ll}
\{k\} & \text{if } 1 \le k < n, \\
\llb k, \infty \rrb & \text{if }k \ge n
\end{array}
\right.
\qquad\text{and}\qquad
\mathscr U_k^{\sf m}(H) = 
\left\{
\begin{array}{ll}
\{k\} & \text{if } 1 \le k \le n, \\
\emptyset & \text{if }k > n.
\end{array}
\right.
\end{equation*}
\end{enumerate*}

\vskip 0.05cm

\begin{enumerate*}[label=\textup{(\arabic{*})}, mode=unboxed, resume]
\item Fix an integer $n \ge 2$ and let $\mathcal P_{{\rm fin}, 0}(\mathbb Z_n)$ be the \evid{reduced power monoid} of the additive group $\mathbb Z_n$ of integers modulo $n$, i.e., the (additively written) monoid obtained by endowing the subsets of $\mathbb Z_n$ containing the zero element $[0]_n \in \mathbb Z_n$ with the operation of setwise addition $(X, Y) \mapsto \{x + y \colon x \in X, \, y \in Y\}$. \\

\indent{}By \cite[Proposition 4.11(ii)]{Tr20(c)}, every $X \in \mathcal P_{{\rm fin}, 0}(\mathbb Z_n)$ factors as a product of irreducibles.
On the other hand, it is obvious that $\{[0]_n, x\}$ is an irreducible of $\mathcal P_{{\rm fin}, 0}(\mathbb Z_n)$ for every non-zero $x \in \mathbb Z_n$. Consequently, we get from the \emph{proof} of \cite[Proposition 4.12(i) and Lemma 5.5]{An-Tr18} (reworked in terms of irreducibles) that every minimal factorization in $\mathcal P_{{\rm fin}, 0}(\mathbb Z_n)$ has length smaller than $n$ and the interval $\llb 2, n-1 \rrb$ is a minimal length set of $\mathcal P_{{\rm fin}, 0}(\mathbb Z_n)$. Since $\mathbb Z_n$ is a proper idempotent of $\mathcal P_{{\rm fin}, 0}(\mathbb Z_n)$, it follows that, for every integer $k \ge 2$,
\begin{equation*}
\mathscr U_k(\mathcal P_{{\rm fin}, 0}(\mathbb Z_n)) = \mathbb N_{\ge 2}
\qquad\text{and}\qquad
\mathscr U_k^{\sf m}(\mathcal P_{{\rm fin}, 0}(\mathbb Z_n)) = 
\left\{
\begin{array}{ll}
\llb 2, n-1 \rrb & \text{if } 2 \le k < n, \\
\emptyset & \text{if }k \ge n.
\end{array}
\right.
\end{equation*}
Note also that $\mathscr U_1(\mathcal P_{{\rm fin}, 0}(\mathbb Z_n)) = \{1\}$ when $n$ is odd, since every irreducible of $H$ is then an atom \cite[Theorem 4.12]{Tr20(c)}; and $\mathscr U_1(\mathcal P_{{\rm fin}, 0}(\mathbb Z_n)) = \mathbb N^+$ when $n$ is even, because in this latter case the set $\{[0]_n, [n/2]_n\}$ is an i\-dem\-po\-tent irreducible of $\mathcal P_{{\rm fin}, 0}(\mathbb Z_n)$.
\end{enumerate*}
\end{example}

\section{An example}
\label{sec:example}

Given a set $X$ and a binary relation $R$ on the free monoid $\mathscr F(X)$, we define $R^\sharp$ as the smallest monoid congruence on $\mathscr F(X)$ containing $R$.
This means that $\mathfrak u \equiv \mathfrak v \bmod R^\sharp$ if and only if there are $\mathfrak z_0, \mathfrak z_1, \allowbreak \ldots, \mathfrak z_n \in \mathscr F(X)$ with $\mathfrak z_0 = \mathfrak u$ and $\mathfrak z_n = \mathfrak v$ such that, for each $i \in \llb 0, n-1 \rrb$, there exist $X$-words $\mathfrak p_i$, $\mathfrak q_i$, $\mathfrak q_i^\prime$, and $\mathfrak r_i$ with the following properties:
\begin{center}
(i) either $\mathfrak q_i = \mathfrak q_i^\prime$, or $\mathfrak q_i \RR \mathfrak q_i^\prime$, or $\mathfrak q_i^\prime \RR \mathfrak q_i$; \hskip 1cm (ii) $\mathfrak z_i = \mathfrak p_i \ast \mathfrak q_i \ast \mathfrak r_i$ and $\mathfrak z_{i+1} = \mathfrak p_i \ast \mathfrak q_i^\prime \ast \mathfrak r_i$.
\end{center}
We denote by $\mathrm{Mon}\langle X \mid R \rangle$ the monoid obtained by taking the quotient of $\mathscr F(X)$ by the congruence $R^\sharp$;
we write $\mathrm{Mon}\langle X \mid R \rangle$ multiplicatively and call it a (\evid{monoid}) \evid{presentation}. We refer to the elements of $X$ as the \evid{generators} of the presentation, to each pair $(\mathfrak q, \mathfrak q^\prime) \in R$ as a \evid{defining relation}, and to each $X$-word in a defining relation as a \evid{defining word}. If there is no danger of confusion, we systematically identify an $X$-word $\mathfrak z$ with its equivalence class in $\mathrm{Mon}\langle X \mid R \rangle$.

The \evid{left graph} of a presentation $\mathrm{Mon}\langle X \mid R\rangle$ is the undirected multigraph with vertex set $X$ and an edge from $y$ to $z$ for each pair $(y \ast \mathfrak u, z \ast \mathfrak v) \in R$ with $y, z \in X$ and
$\mathfrak u, \mathfrak v \in \mathscr F(X)$; this results in a loop when $y = z$, and in multiple (or parallel) edges between $y$ and $z$ if there are two or more defining relations of the form $(y \ast \mathfrak u, z \ast \mathfrak v)$.
The \evid{right graph} is defined analogously, using the
right-most (instead of left-most) let\-ters of each word from a defining relation. The left and the right graphs of a presentation were first considered by Adian \cite{Ad66}, whence we refer to them as the \evid{Adian graphs} of $\mathrm{Mon}\langle X \mid R\rangle$. 

On the other hand, a \evid{piece} of a presentation $\mathrm{Mon}\langle X \mid R\rangle$ is a non-empty $X$-word $\mathfrak u$ for which there exist
$\mathfrak p, \mathfrak p^\prime, \mathfrak q, \mathfrak q^\prime \in \mathscr F(X)$ with $\mathfrak p \ne \mathfrak p^\prime$ or 
$\mathfrak q \ne \mathfrak q^\prime$ such that $
\mathfrak v := \mathfrak p \ast \mathfrak u \ast \mathfrak q$ and $\mathfrak v' := \mathfrak p^\prime \ast \mathfrak u \ast \mathfrak q^\prime$ are defining words; in particular, it is not required that $\mathfrak v \ne \mathfrak v'$ or $(\mathfrak v, \mathfrak v') \in R$. The notion was first conceived in the study of \emph{group} presentations and later extended by Kashintsev to semigroups (see \cite[Section 1]{Ka92}).

Following \cite{Ka92}, we say that a monoid $H$ is \evid{of class $K_p^q$}, for some $p, q \in \mathbb{N}^+$, if $H$ is isomorphic to a monoid presentation $\textup{Mon} \langle X \mid R \rangle$ with finitely many generators such that (i) no defining word can be expressed as the concatenation in $\mathscr F(X)$ of less than $p$ pieces and (ii) the Adian graphs of the presentation have both girth $\ge q$ (we recall that the \evid{girth} of an undirected multigraph $G$ is the shortest length of a cycle in $G$, with the understanding that the girth of a cycle-free multigraph is $\infty$). 
Our interest in these definitions is linked to the following result, first proved by Guba in \cite[Theorem 1]{Gu94b} and hence referred to as \emph{Guba's} (\emph{embedding}) \emph{theorem} (note that, in \cite{Gu94b}, there is a typo in the very def\-i\-ni\-tion of a piece, as discussed on MathOverflow at \url{https://mathoverflow.net/questions/353340/}).

\begin{theorem}\label{thm:guba-embedding-thm}
Every monoid of class $K_3^2$ embeds into a group \textup{(}and hence is cancellative\textup{)}.
\end{theorem}

Guba's theorem will come in handy in the next example, which shows that Theorem \ref{thm:main} does not carry over to the non-commutative setting in any obvious way. 

\begin{example}\label{exa: non-comm counterexample}

Let $H$ be the presentation $\textup{Mon} \langle A \mid R \rangle$, where $A$ is the $3$-element set $\{a, b, c\}$ and $R$ is the set $\{(\mathfrak s_n, \mathfrak t_n) \colon n = 2, 3, \ldots\} \subseteq \mathscr F(A) \times \mathscr F(A)$, with
\[
\mathfrak s_n := c \ast a^{\ast n} \ast b^{\ast 2^n} \ast a^{\ast n} \ast c 
\quad\textup{and}\quad
\mathfrak t_n := a \ast c^{\ast n} \ast b^{\ast n} \ast c^{\ast n} \ast a.
\]
It is routine to check that $H$ is a $3$-generated monoid with trivial group of units whose $\mid_H$-irreducibles are $\underline{a}$, $\underline{b}$, and $\underline{c}$, where we write $\underline{\mathfrak u}$ for the $R^\sharp$-congruence class (in $H$) of an $A$-word $\mathfrak u$. Each of $\underline{a}$, $\underline{b}$, and $\underline{c}$ is on the other hand an atom, because the defining words in $R$ have all length $\ge 2$. Then, $H$ is an atomic monoid and every $\mid_H$-factorization is an atomic factorization and vice versa.
In addition, it is immediate that, for each $k \in \mathbb{N}^+$, none of the $A$-words
\begin{equation}\label{equ:pieces}
c \ast a^{\ast k} \ast b, 
\quad
b \ast a^{\ast k} \ast c,
\quad
a \ast c^{\ast k} \ast b,
\quad\textup{or}\quad
b \ast c^{\ast k} \ast a
\end{equation}
is a piece of $\textup{Mon} \langle A \mid R \rangle$; if, e.g., a defining word factors as $\mathfrak p \ast (a \ast c^{\ast k} \ast b) \ast \mathfrak q$ for some $A$-words $\mathfrak p$ and $\mathfrak q$, then necessarily $\mathfrak p = \varepsilon_A$ and $\mathfrak q = b^{\ast (k-1)} \ast c^{\ast k} \ast a$ (and the other cases are similar).
Since a non-empty $A$-word $\mathfrak s := x_1 \ast \cdots \ast x_l$ is a piece of $\textup{Mon} \langle A \mid R \rangle$ only if so is any subword of the form $x_i \ast \cdots \ast x_j$ with $1 \le i \le j \le l$, it follows that the support $\{x_1, \ldots, x_l\}$ of $\mathfrak s$ is a \emph{proper} subset of $A$. As a matter of fact, a piece is in the first place a subword of a defining word, so that the shortest pieces with support $A$ (if there were any) would be those listed in Eq.~\eqref{equ:pieces}.

It follows that no defining word of $\textup{Mon} \langle A \mid R \rangle$ is the concatenation in $\mathscr F(A)$ of less than three pieces (note that each defining word is palindromic and its support is $A$). Consequently, $H$ is a monoid of class $K_3^2$, which implies, by Guba's theorem, that $H$ is cancellative.
So, it remains to see that, for each $n \in \mathbb{N}^+$, there are minimal atomic factorizations $\mathfrak a_n$ and $\mathfrak b_n$ of an element $x_n \in H$ such that~$\|\mathfrak a_n\| \ge n\, \|\mathfrak b_n\|$.

For, fix an integer $n \ge 2$. The $H$-words $
\mathfrak a_n :=  \underline{c} \ast \underline{a}^{\ast n} \ast \underline{b}^{\ast 2^n} \ast \underline{a}^{\ast n} \ast \underline{c}$ and $\mathfrak b_n := \underline{a} \ast \underline{c}^{\ast n} \ast \underline{b}^{\ast n} \ast \underline{c}^{\ast n} \ast \underline{a}$
are non-empty atomic factorizations of the same element,
with the further property that $\|\mathfrak b_n\|^{-1} \|\mathfrak a_n\| \to \allowbreak \infty$ as $n \to \infty$.
Therefore, it suffices to check that $\mathfrak a_n$ and $\mathfrak b_n$ are minimal atomic factorizations.

Assume to the contrary that $\mathfrak a_n$ is not a minimal atomic factorization (the minimality of $\mathfrak b_n$ can be proved in a similar fashion). Then $\mathfrak s_n \equiv \allowbreak \mathfrak u \bmod R^\sharp$ for some permutation $\mathfrak u$ of a \emph{proper} subword of $\mathfrak s_n$; that is, there exist a smallest $k \in \mathbb N^+$ and $A$-words $\mathfrak z_0, \mathfrak z_1, \allowbreak \ldots, \mathfrak z_k$ with $\mathfrak z_0 = \allowbreak \mathfrak s_n$ and $\mathfrak z_k = \mathfrak u$ such that, for each $i \in \llb 0, k-1 \rrb$, there are an integer $m_i \ge 2$ and $\mathfrak p_i, \mathfrak r_i \in \mathscr F(A)$ with 
\begin{center}
(i) $\mathfrak z_i = \mathfrak p_i \ast \mathfrak s_{m_i} \ast \mathfrak r_i$ and $\mathfrak z_{i+1} = \mathfrak p_i \ast \mathfrak t_{m_i} \ast \mathfrak r_i$, \hskip 0.5cm \text{or} \hskip 0.5cm (ii) $\mathfrak z_i = \mathfrak p_i \ast \mathfrak t_{m_i} \ast \mathfrak r_i$ and $\mathfrak z_{i+1} = \mathfrak p_i \ast \mathfrak s_{m_i} \ast \mathfrak r_i$.
\end{center}
However, it is readily seen (by induction on $i$) that this can only happen if $\mathfrak z_i \in \{\mathfrak s_n, \mathfrak t_n\}$ for every $i \in \llb 0, k \rrb$, because, for an integer $m \ge 2$ with $m \ne n$, neither $\mathfrak s_m$ nor $\mathfrak t_m$ is a divisor of $\mathfrak s_n$ (resp., of $\mathfrak t_n$) in $\mathscr F(A)$. It follows that $\mathfrak u = \mathfrak t_n$, which is impossible since neither $\mathfrak s_n$ nor $\mathfrak t_n$ is a proper subword of $\mathfrak s_n$.
\end{example}

\section{Prospects for future research}
\label{sec:closings}

By Theorem \ref{thm:main-premon}, the minimal elasticity $\rho^{\rm m}(H)$ of (the divisibility premon of) a cancellative, commutative, reduced monoid $H$ with finitely many irreducibles is finite. In view of \cite[Proposition 3.4(1)]{FGKT}, we are led to ask if $\rho^{\rm m}(H)$ is, in fact, a rational number (cf.~the last lines of Sect.~\ref{sec:elasticity}).

On the other hand, we know from Sect.~\ref{sec:example} that the minimal elasticity of an atomic, cancellative, finitely generated, reduced monoid need not be finite. Is this still the case with an \evid{acyclic}, finitely generated, and reduced monoid? Here, we say a monoid $H$ is \evid{acyclic} if $uxv \ne x$ for all $u, v, x \in H$ such that $u$ or $v$ is a non-unit. Acyclic monoids were introduced in \cite[Definition 4.2]{Tr20(c)} and further studied in \cite{Co-Tr-22(a)}; from an arithmetical point of view, they provide a better alternative to cancellativity in the non-commutative setting. For instance, we get from \cite[Corollary 4.4]{Tr20(c)} that, in an acyclic monoid, irreducibles and atoms are the same thing, which generalizes an observation made in the first lines of Sect.~\ref{sec:intro}. It is also worth noting that every acyclic monoid is unit-cancellative; the two notions (of acyclicity and unit-cancellativity) coincide in the commutative setting, but not in general \cite[Example 4.8]{Tr20(c)}. 

\section*{Acknowledgments}
L.\,C.~was supported by the European Union's Horizon 2020 research and innovation programme under the Marie Sk\l{}odowska-Curie grant agreement No.~101021791, and by the FWF (Austrian Science Fund) project No.~P33499-N. The Marie Sk\l{}odowska-Curie grant also financed S.\,T.'s visit at University of Graz in summer-fall 2022, during which the paper was written. The authors are both members of the Gruppo Nazionale per le Strutture Algebriche, Geometriche e le loro Applicazioni (GNSAGA) of the Italian Mathematics Research Institute (INdAM).

\end{document}